\documentclass[11pt]{amsart}
\usepackage{amsmath,amsfonts,latexsym,amssymb,amsthm,mathrsfs,array}

\usepackage{amsfonts}

\newcommand{\CC}{{\mathbb C}}
\newcommand{\ZZ}{{\mathbb Z}}
\newcommand{\NN}{{\mathbb N}}
\newcommand{\I}{{\mathcal{I}}}
\def\gg{{\mathfrak{g}}}
\def\hh{{\mathfrak{h}}}
\def\CC{{\mathbb C}}
\def\ZZ{{\mathbb Z}}
\def\TT{{\mathcal{T}}}
\newcommand{\fma}{\overset{\circ}{\mathfrak{g}}}
\newcommand{\fmh}{\overset{\circ}{\mathfrak{h}}}

\newcommand{\C}{{\mathbb C}}
\newtheorem{dfn}{Definition}[section]
\newcommand{\bdfn}{\begin{dfn}\rm}
\newcommand{\edfn}{\end{dfn}}
\newtheorem{thm}[dfn]{Theorem}
\newcommand{\bthm}{\begin{thm}}
\newcommand{\ethm}{\end{thm}}                   
\newtheorem{lmma}[dfn]{Lemma}                   
\newcommand{\blmma}{\begin{lmma}}                   
\newcommand{\elmma}{\end{lmma}}                   
\newtheorem{ppsn}[dfn]{Proposition}
\newcommand{\bppsn}{\begin{ppsn}}
\newcommand{\eppsn}{\end{ppsn}}
\newtheorem{crlre}[dfn]{Corollary}
\newtheorem{rmk}[dfn]{Remark}
\newcommand{\brmk}{\begin{rmk}\rm} 
\newcommand{\ermk}{\end{rmk}}

\def\I{{\mathcal I}}

\def\SS{{\mathcal S}}
\def\LL{{\mathfrak L}}

\def\HH{{\mathcal H}}
\numberwithin{equation}{section}

\title{Integrable modules for twisted toroidal extended affine Lie algebras}
\author{ S. Eswara Rao, Sachin S. Sharma and Punita Batra}
\date{}
\begin{document}
\maketitle
\begin{abstract}
In this paper we classify the irreducible integrable modules for the twisted toroidal extended affine Lie algebras (twisted toroidal EALA, in short) with finite dimensional weight spaces when the finite dimensional center acts non-trivially.  Using an automorphism we reduce to the case where $K_0$ acts non-trivially and $K_i$, $1 \leq i \leq n$ act trivially.  Twisted toroidal EALA has natural triangular decomposition and we prove that any irreducible integrable module of it with finite dimensional weight spaces is a highest weight module with respect to the above triangular decomposition. The highest weight space is an irreducible module for the zeroth  component of the twisted toroidal EALA. We then describe the highest weight space in detail.

\end{abstract}

\section{Introduction}
Extended affine Lie algebras (EALAs, for short) are natural generalization of affine Kac-Moody algebras; finite dimensional simple Lie algebra, extended toroidal Lie algebras are examples of EALAs. They first appeared in the work of Saito and Slodowy on elliptic singularities and in the paper of physicists H$\o$egh and Torresani in 1990. Later mathematicians like 
Allison, Azam, Berman, Gao, Pianzola, Neher  and Yoshii systematically developed the theory of EALAs (see \cite{ABFP, EN, YOSH}, and references therein). Unlike affine Kac-Moody algebras, EALAs may have infinite dimensional centers. So representation theory of EALA's is still in progress.

Let $L(\fma) = \fma \otimes A_n$ is a loop algebra, where $\fma$ is a finite dimensional simple Lie algebra over $\mathbb{C}$ and $A_n$ is a  commutative Laurent polynomial ring of $n$ variables. Let $L(\fma) \oplus \Omega_{A_n}/d{A_n}$ be universal central extension of $L(\fma)$ (note that the center is infinite dimensional for $n \geq 2$). These are called toroidal Lie algebras (TLA, in short) and they have been extensively studied \cite{R, ESW}. Let $\text{Der}(A_n)$ be the Lie algebra of darivations on $A_n$ and consider the full toroidal Lie algebra (FTLA, in short) $L(\fma) \oplus \Omega_{A_n}/d{A_n} \oplus \text{Der}(A_n)$. Representation of FTLA also have been studied \cite{RJ}. Neither TLA  nor FTLA are EALA as they do not admit non-degenerate symmetric invariant bilinear form. To rectify this one considers subalgebra  $\SS_n$ of 
$\text{Der}(A_n)$ consisting of divergence zero vector fields (see Section \ref{sec4} for more details). Then the Lie algebra $\tau_{\text{div}} = L(\fma) \oplus \Omega_{A_n}/d{A_n} \oplus \SS_n$
is called toroidal extended affine Lie algebras. These are in fact EALAs. In \cite{BY} Billig constructed irreducible representation of toroidal EALAs using
vertex operator algebras.  In \cite{BLAU} Billig and Lau constructed irreducible modules for twisted toroidal EALAs. In \cite{ FZT1, FZT}, full classification of irreducible integrable modules with finite dimensional weight spaces of toroidal EALA for $n = 2$ was obtained. In this paper we assume $n \geq 3$ and classify irreducible integrable modules for toroidal EALAs where center acts non-trivially (see definition in Section \ref{sec4}).  Actually we consider more general EALAs. We consider commutating finite order automorphisms $\sigma_0, \ldots, \sigma_n$ of
$\fma$ and consider the multi-loop algebra $\bigoplus_{(k_0,k) \in \ZZ^{n+1}}{\fma(\bar{k}_0, \bar{k})t_0^{k_0}t^{k}}$. We assume that multi-loop algebra is a Lie torus (see Section \ref{sec4} for definition). This is not a very serious assumption as it is well known that centerless cores of almost all EALAs
are Lie tori \cite{ABFP}.  We then consider universal central extension of multi-loop algebra and add
$\SS_{n+1}$ the Lie algebra consisting of divergence zero vector fields on $A_{n+1}$ (see Section \ref{sec4} for details).  The resulting Lie algebra is EALA called as twised toroidal EALA, and we denote it by $\TT$.   The aim of this paper is to classify all  irreducible integrable representation of twisted toroidal EALAs with finite dimensional weight spaces where center acts non-trivially.

The proof of our classification problem runs parallel to \cite{PSER} where they consider full twisted TLA. Our twisted toroidal EALA is a proper subalgebra of full twisted EALA. Since we working with smaller algebra
some of the results used in \cite{PSER} won't work here. So we use differnt techniques; notably results
from \cite{BT} and \cite{GL} are very special for divergence zero vector fields.

We now explain our plan of classification result in more detail. In Section \ref{sec4} we start with basic definitions which leads up to the definition of twisted toroidal EALA $\TT$.  In Section \ref{sec5} we 
define root space decomposition and using it define a natural triangular decomposition $\TT = \TT^{-} \oplus \TT^{0} \oplus \TT^{+}$. We fix an irreducible integrable module $V$ of $\TT$ with finite dimensional weight spaces where $K_0$ acts non-trivially and $K_i$ act trivially for $1 \leq i \leq n$. This
assumption is due to the fact that  group $GL(n+1, \ZZ)$ acts as automorphism on FTLA and leaves $\tau_{\text{div}}$ invariant. Then using earlier work of \cite{PSER} it follows that the space $M = \{v \in V | \TT^{+}v = 0\}$ is a non-zero $\TT^0$ irreducible module.  The rest of the paper revolves around $M$ as we try to decode the structure of $M$. The $\TT^0$-module $M$ has natural $\ZZ^n$
gradation and using integrality and Weyl group invariance of weights of $V$, we prove the dimensions of graded pieces of $M$ are uniformly bounded.  Using this we prove that $\Omega_{A(m)}/d_{A(m)}$ acts trivially on $M$(Propostion \ref{thc}). Then we prove proposition
\ref{prpm} which encodes the action subspace $\displaystyle{\sum_{r \in \Gamma}}{\CC t^{r}K_0} \oplus \displaystyle{\sum_{r \in \Gamma}}{\CC t^{r}d_0} $ of $\TT^{0}$. We would like to emphasize the fact unlike \cite{PSER} where the corresponding results come by earlier works of authors, we need completely different techinque to prove Proposition \ref{prpm}. For its proof 
hinges on results from \cite{GL} (Theorem \ref{thi} in our paper) and \cite{CP}. By Propsition \ref{prpm}
it follows that $M \cong M^1 \otimes A(m)$ for some space $M^1$. Then for some suitable subspace $W$ of $M$, we consider $\widetilde{V} = M/W$. Here $W$ and $\widetilde{V}$ are not $\TT^0$- modules but modules for some suitable Lie algebra $\mathcal{I}$ (see Sections \ref{sec5} and \ref{sec7} for details ). It follows from
the work of \cite{PSER} that  $\widetilde{V}$ is completely reducible $\mathcal{I}$-module with isomorphic irreducible components. In Section \ref{sec7},  using results of \cite{NSS} and\cite{BT} we classify all the finite dimensional irreducible modules for $\mathcal{I}$. Finally in Section \ref{sec8}
by going up procedure we define module for $\TT^0$, which contains an isomorphic copy of $M$
and we conclude with the final result Theorem \ref{THM}.

\section{} \label{sec4}
Let $A_{n+1} = \CC[t_0^{\pm 1}, \cdots, t_n^{\pm 1}]$ be a Laurent polynomial ring in $n + 1$ variables.
$L(\fma) = \fma \otimes A_{n+1}$ be the corresponding loop algebra, $\fma$ is a finite dimensional simple Lie 
algebra over $\mathbb{C}$ with a Cartan subalgebra $\fmh$. Let $\Omega_{A_{n+1}}$ be a vector space spanned by the symbols $t_0^{k_0}t^{k}K_i, 0\leq i \leq n,
k_0 \in \ZZ, k \in \ZZ^{n}$. Let $d{A_{n+1}}$ be the subspace spanned by $\sum_{i =0}^{n}{k_i t_0^{k_0}t^{k}K_i}$. It is well known that ${\tilde{L}}(\fma) = L(\fma) \oplus \Omega_{A_{n+1}}/d{A_{n+1}}$ is the
universal central extension of $L(\fma)$ with the following brackets:

$[X(k_0, k), Y(l_0, l)] = [X,Y](l_0+k_0, l+k) + (X|Y) \sum_{i =0}^{n}{k_i t_0^{l_0 + k_0}t^{l + k}K_i}$,
where $X(k_0, k) = X \otimes t_{0}^{k_0}t^{k}$. Recall the Lie algebra of derivations of $A_{n+1}$,
$\mathrm{Der}(A_{n+1})$ with basis $\{d_i, t_{{0}}^{r_0} t^{r}d_{i} | 0 \leq i \leq n, 0 \neq (r_0, r) \in \ZZ^{n+1}\}$.  $\mathrm{Der}(A_{n+1})$ acts on $ \Omega_{A_{n+1}} /d{A_{n+1}}$ by 

$ t_{{0}}^{p_0} t^{p}d_{a}( t_{{0}}^{q_0} t^{q} K_{b}) = q_{a} t_{{0}}^{p_0 + q_0} t^{p + q} K_{b}
+ \delta_{ab} \sum_{c=0}^{n}{p_{c} t_{{0}}^{p_0 + q_0} t^{p + q}K_{c}}.$

There are two non-trivial 2-cocycle of $\mathrm{Der}(A_{n+1})$ with values in $\Omega_{A_{n+1}}/d {A_{n+1}}$:

$\phi_{1}( t_{{0}}^{p_0} t^{p}d_{a},  t_{{0}}^{q_0} t^{q} d_{b}) = -q_{a}p_{b} \sum_{c=0}^{n}{p_{c} t_{{0}}^{p_0 + q_0} t^{p + q}K_{c}}$, 

$\phi_{1}( t_{{0}}^{p_0} t^{p}d_{a},  t_{{0}}^{q_0} t^{q} d_{b}) = -p_{a}q_{b} \sum_{c=0}^{n}{p_{c} t_{{0}}^{p_0 + q_0} t^{p + q}K_{c}} .$
Let $\phi$ be any linear combination of $\phi_1$ and $\phi_2$. Then 
$\tau = L(\fma) \oplus \Omega_{A_{n+1}}/d{A_{n+1}} \oplus \mathrm{Der}(A_{n+1})$ is a Lie algebra with the following brackets and called as full toroidal Lie algebra (FTLA, for short):
\begin{align*}
[t_{{0}}^{p_0} t^{p}d_a, X(q_0, q)] & = q_{a} X(p_0 + q_0, p+q),\\
[t_{{0}}^{p_0} t^{p} d_{a}, t_{{0}}^{q_0} t^{q} K_{b}] & = q_a t_{{0}}^{p_0 + q_0} t^{p + q}K_b +
\delta_{ab}\sum_{c=0}^{n}{p_{c} t_{{0}}^{p_0 + q_0} t^{p + q}K_{c}}, \\
[ t_{{0}}^{p_0} t^{p} d_{a}, t_{{0}}^{q_0} t^{q} d_{b}] & = q_{a} t_{{0}}^{p_0 + q_0} t^{p + q} -
p_{b} t_{{0}}^{p_0 + q_0} t^{p + q} + \phi(t_{{0}}^{p_0} t^{p} d_{a}, t_{{0}}^{q_0} t^{q} d_{b}).
\end{align*}
Now consider the subalgebra of divergence zero vector fields $\mathcal{S}_{n+1}$ of $\mathrm{Der}(A_{n+1})$. One can define
$\mathcal{S}_{n+1} = \{D(u,r)| (u|r) = 0,  u \in \CC^{n+1}, r\in \ZZ^{n+1}\}$.

Now consider the subalgebra
$\tau_{\mathrm{div}} = L(\fma) \oplus \Omega_{A_{n+1}}/d{A_{n+1}} \oplus \mathcal{S}_{n+1}$  of $\tau$.
It is well known that unlike $\tau$, $\tau_{\mathrm{div}}$ possesses a non-degenerate symmetric, invariant bilinear form and is called as toroidal extended affine Lie algebra. The form is defined as
follows:
\begin{align*}
\big(X \otimes t^{r} \rvert Y\otimes t^{s} \big) &= \delta_{r, -s}(X |Y), X,Y \in \fma, r, s \in \ZZ^{n+1};\\
\Bigg(\sum_{c =0}^{n}{a_{c}t^{r} d_{c} \vert t^{s}K_{d}}\Bigg) & = \delta_{r, -s} a_{d}.
\end{align*}
All other brackets of bilinear form are zero. 
    Now let $\gg_1$ be any arbitrary finite dimensional simple Lie algebra $\mathbb{C}$ with a Cartan subalgebra $\hh_1$.  Let $\Delta(\gg_1, \hh_1) = \mathrm{supp}_{\hh_1}{(\gg_1)}$. Then $\Delta_{1}^{\times} = \Delta^{\times}(\gg_1, \hh_1) = \Delta(\gg_1, \hh_1) \backslash \{0\}$ is an irreducible reduced finite
root system with at most two root lengths. Let ${\Delta_{1, \mathrm{sh}}^{\times}}$ denote the non zero short roots of $\Delta_1$. Define

\[
    {\Delta_{1, \mathrm{en}}^{\times}}= 
\begin{cases}
    \Delta_{1}^{\times} \cup 2  {\Delta_{1, \mathrm{sh}}^{\times}}\,\,\,\, &\text{if} \,\,\,\,{\Delta_{1}}^{\times} = B_l \,\, \text{type}\\
   \Delta_{1}^{\times}             & \text{otherwise}.
\end{cases}
\]

 We need the following definition:
\bdfn
A finite dimensional $\gg_1$-module $V$ is said to satisfy condition $(M)$ if $V$ is irreducible
with dimension greater than 1 and weights of $V$ relative to $\hh_1$ are contained in $ { \Delta_{1, \mathrm{en}}^{\times}}$.
\edfn

Now recall that $\fma$ is a finite dimensional simple Lie algebra with a Cartan subalgebra $\fmh$  and let $\sigma_0, \sigma_1, \cdots, \sigma_n $ be the 
commuting automorphism of $\fma$ of order $m_0, m_1, \cdots, m_n$ respectively. Let $m = (m_1, \ldots, m_n) \in \ZZ^n$ . Define
$\Gamma_0  = m_0 \ZZ$ and $\Gamma = m_1 \ZZ \oplus \cdots \oplus m_n \ZZ$. Let $\Lambda_0: = \ZZ/\Gamma_0$ and $\Lambda := \ZZ^{n}/\Gamma$.
Then we have $\fma = \displaystyle{\bigoplus_{(\bar{k}_0,\bar{k})}{\fma(\bar{k}_0,\bar{k}})}$, where $(\bar{k}_0, \bar{k}) = (\bar{k}_0, \bar{k}_1\ldots, \bar{k}_n) \in \Gamma_0 \oplus \Gamma$ , and $\fma(\bar{k}_0,\bar{k}) = \{X \in \fma | \sigma_i (X) = \xi_i^{k_i} X, 0\leq i \leq n\}$, where
$\xi_i$ are $m_i$-th primitive root of unity for $i = 0 \ldots, n$.

\bdfn
A multiloop algebra $\bigoplus_{(k_0,k) \in \ZZ^{n+1}}{\fma(\bar{k}_0, \bar{k})t_0^{k_0}t^{k}}$ is called a Lie torus ${LT}$
if
\begin{enumerate}
\item $\fma(\bar{0},\bar{0})$ is a finite dimensional simple Lie algebra.
\item For $(\bar{k}_0, \bar{k}) \neq (0,0)$ and $\gg(\bar{k}_0, \bar{k}) \neq 0$, $\gg(\bar{k}_0, \bar{k}) \cong U(\bar{k}_0, \bar{k}) \oplus W(\bar{k}_0, \bar{k})$, where $U(\bar{k}_0, \bar{k})$
is trivial as $\gg(\bar{0}, \bar{0})$-module and either $W(\bar{k}_0, \bar{k})$ is zero or satisfy the condition (M).
\item The order of the group generated by $\sigma_i$, $0 \leq i \leq n$ is equal to the product of 
orders of each $\sigma_i$, for $0 \leq i \leq n$.
\end{enumerate}
\edfn
Let $\fmh(0)$ denote a Cartan subalgebra of $\fma(\bar{0}, \bar{0})$. Then by \cite{NAOI} Lemma 3.1.3, $\fmh(0)$ is ad-diagonalizable on $\fma$ and $\Delta^{\times} = \Delta^{\times}(\fma, \fmh(0))$ is an irreducible finite root system in $\fmh(0)$ (Proposition 3.3.5, \cite{NAOI}). Let $\Delta_{0} := \Delta(\fma(\bar{0}, \bar{0}), \fmh(0))$
One of the main properties of Lie tori is that  $\Delta := \Delta(\gg, \fmh(0)) = \Delta_{0, \mathrm{en}}$ (Proposition 3.2.5, \cite{ABFP}).
Let $A(m) = \CC[t_{1}^{\pm m_1}, \ldots, t_n^{\pm m_n}]$ and $A(m_0, m) = 
\CC[t_{0}^{\pm m_0}, t_{1}^{\pm m_1}, \ldots, t_n^{\pm m_n}]$ and
$\mathcal{S}_{n+1}(m_0, m) = \{D(u,r)\,|\, (u | r) =0, u \in \CC^{n+1}, r \in \Gamma_0 \oplus \Gamma\}$. 
The Lie algebra $\mathcal{T} = LT \bigoplus \Omega_{A(m_0,m)}/d_{A(m_0, m)} \bigoplus \mathcal{S}_{n+1}(m_0,m)$
is called twisted toroidal extended affine Lie algebra . The purpose of this paper is to classify the irreducible integrable modules of twisted toroidal EALAs. 

{\bf{Automorphism}}: Let $B = (b_{ij}) \in GL(n+1, Z)$ and $B^{-1} = (c_{ij})$, then the following is an automorphism $\phi_{B}$ on $\tau$ by:
\begin{align*}
\phi_{B} (X \otimes t^{(k_0, k)} )& = X \otimes t^{(k_0,k)B^{t}},\\
\phi_{B}(t^{(k_0, k)}K_i) &= \displaystyle{\sum_{p =0}^{n}{b_{ip}}t^{(k_0, k)B^{t}}},\\
\phi_{B}(t^{(k_0, K)} d_j) & = \displaystyle{\sum_{p =0}^{n}{c_{jp}}t^{(k_0, k)B^{t}}}. 
\end{align*}
 It is easy to check that $\phi_B$ is an automorphism on $\tau$. Now as for $u \in \CC^{n+1}, r \in \ZZ^{n+1} $ such that $(u|r) =0$, $\phi_B D(u,r) = D(uB^{-1}, r B^{t})$ and as $(u B^{-1} | r B^{t}) = (u| r B^{t} {B^{-1}}^t) =0$. So $\SS_{n+1}$ is invariant under $\phi_B$. Hence $\phi_B$ leaves $\tau_{\text{div}}$
invariant. So up to an automorphism we can assume that the central element $K_0$ acts as $C_0 \in \ZZ_{>0}$ and $K_i$'s as $0$ for $1 \leq i \leq n$.

\section{} \label{sec5}
In this section we will do root decomposition of $\TT$. Let $\HH = \fmh(0) \oplus \displaystyle{\sum_{i =0}^{n}\CC K_i \oplus \sum_{i=0}^{n}} \CC d_i$ will be our Cartan subalgebra for the root decomposition of $\TT$. Define $\delta_i, \omega_i \in \HH^*$ for $0 \leq i \leq n$ by $\delta_i (\fmh(0)) =0,  \delta_i (K_j) = 0, \delta_i (d_j) = \delta_{i j}$ and $\omega_i (\fmh(0)) = 0, \omega_i (K_j) = \delta_{i j}, \omega_i (d_j) = 0$. Let us denote $\delta_{k} = \displaystyle{\sum_{i =1}^{n}{k_i \delta_i}}$ for $k \in \ZZ^n$. Define $\fma(\bar{k}_0, \bar{k}, \alpha) := \{ x \in \fma(\bar{k}_0, \bar{k}) | [h, x] = \alpha(h)x, \,\,  \forall \,\, h \in \fmh(0)\}$. Then we have 
$\TT = \displaystyle{\bigoplus_{\beta \in \Delta}}{\TT_{\beta}}$, where $\Delta \subseteq \{ \alpha + k_0 \delta_0 + \delta_k | \alpha \in \Delta_{0, \mathrm{en}}, k_0 \in \ZZ, k \in \ZZ^{n} \}$.  We have
$\TT_{\alpha + k_0 \delta_0 + \delta_k} = \fma(\bar{k}_0, \bar{k}, \alpha) \otimes t_{0}^{k_0}t^{k}$ for $\alpha \neq 0$, and  $\TT_{ k_0 \delta_0 + \delta_k} = \fma(\bar{k}_0, \bar{k}, 0) \otimes t_{0}^{k_0}t^{k} \oplus  \displaystyle{\bigoplus _{i = 0}^{n}{\CC t_{0}^{k_0}t^{k}}K_i  \oplus \bigoplus_{i =0}^{n} \CC t_{0}^{k_0}t^{k} d_i}$
for  $k_0 \delta_0 + \delta_k \neq 0$ and $\TT_0 = \HH$.

In order to define a non-degenerate form on $\HH^*$ we first extend $\alpha \in \fmh(0)^*$ to $\HH$ by defining
$\alpha(K_i) = \alpha(d_i) = 0, \,\, \forall \,\, 0 \leq i \leq n$. Then $(\fmh(0) | K_i) =  (\fmh(0) | d_i), (\delta_k + \delta_{k_0} | \delta_l + \delta_{l_0}) =  (\omega_i | \omega_j) = (\delta_k, \delta_l) = 0$ and $(\delta_i | \omega_j) = \delta_{i j}$, and form on $\fmh (0)$  is restriction of the form of $\fma$. It is easy to see that this form is non-degenerate on $\HH^*$. A root $\beta = \alpha + k_0 \delta_0 +
\delta_k$ is called real root if $\alpha \neq 0$. Let $\Delta^{\mathrm{re}}$ denote the set of all real roots,
and $\beta^{\vee} = \alpha^{\vee} + \frac{2}{(\alpha|\alpha)} \displaystyle{\sum_{i =0}^{n}{k_i K_i}}$ is
co-root of $\beta$, where $\alpha^{\vee}$ is co-root of $\alpha \in \Delta_{0, \mathrm{en}}$. Then
$\beta(\beta^{\vee}) = \alpha(\alpha^{\vee}) = 2$.  For any real root $\gamma$ of $\TT$, define
the reflection $r_{\gamma}$ on $\HH^*$ by $r_{\gamma}(\lambda) = \lambda - \lambda(\gamma^{\vee}) \gamma$ for $\lambda \in \HH^*$. Let Weyl group $W$ of $\TT$ be the 
group generated by $r_{\gamma}, \,\, \forall \,\,\gamma \in \Delta^{\mathrm{re}}$. Now we define the category of integrable modules for $\TT$.
\bdfn
A $\TT$-module $V$ is called integrable if 
\begin{enumerate}
\item $\displaystyle{V = \bigoplus_{\lambda \in \HH^* } {V_{\lambda}}}$, where $V_{\lambda} = \{v \in V |\,\, h.v = \lambda(h)v \,\, \forall \,\, h \in \HH\}$ and $\mathrm{dim}(V_{\lambda}) < \infty$.
\item All the real root vectors act locally nilpotently on $V$, i.e., $\fma(\bar{k}_0, \bar{k}, \alpha) \otimes t_{0}^{k_0}t^{k}$ acts locally nilpotently on $V$ for all $0 \neq \alpha \in \Delta_{0,\mathrm{en}}$.
\end{enumerate}
\edfn
We have the following :
\bppsn
Let $V$ be an irreducible integrable module for $\TT$. Then
\begin{enumerate}
\item $P(V) = \{\gamma \in \HH^* |\,  V_{\gamma} \neq 0\}$ is $W$- invariant.
\item $\mathrm{dim}(V_{\gamma}) = \mathrm{dim}(V_{w \gamma}), \,\, \forall \,\, w \in W$.
\item If $\lambda \in P(V)$  and $\gamma \in \Delta^{\mathrm{re}}$, then $\lambda (\gamma^{\vee}) \in \ZZ$.
\item If $\lambda \in P(V)$  and $\gamma \in \Delta^{\mathrm{re}}$, and $\lambda(\gamma^{\vee}) >0$, then $\lambda - \gamma \in P(V)$.
\item For $\lambda \in P(V)$, $\lambda(K_i)$ is an integer independent of $\lambda$.
\end{enumerate}
\eppsn

\subsection{}
Now rest of this paper we will work with irreducible integrable representation $V$ of $\TT$ with
central element $K_0$ acts as $C_0 \in \ZZ_{>0}$ and rest of $K_i$'s act trivially for $1 \leq i \leq n$.
We now define an order on $\HH^*$. Let for $\lambda \in \HH^*$, $\lambda'$ denote the restriction of $\lambda$ to $\fmh(0)^*$ and for a given $\lambda '\in \fmh(0)^*$, extend it to $\HH^*$ by
defining $\lambda'(K_i) =0$ and $\lambda'(d_i) =0$ for $0 \leq i \leq n$. Then we have unique expression for $\lambda = \displaystyle{\lambda' + \sum_{i =0}^{n}{\lambda(K_i) \omega_i} + \sum_{i=0}^{n}{\lambda(d_i) \delta_i}}$ and for $\lambda \in P(V)$ we have $\lambda = \lambda' +
\lambda(K_0)\omega_0 + \lambda(d_0)\delta_0 +\displaystyle{ \sum_{i=1}^{n}{\lambda(d_i) \delta_i}}.$
So for $\lambda \in P(V)$, we write $\lambda = \bar{\lambda} + \displaystyle{ \sum_{i=i}^{n}{\lambda(d_i) \delta_i}}$, where $\bar{\lambda} = \lambda' +
\lambda(K_0)\omega_0 + \lambda(d_0)\delta_0$. Let $\beta_0$ be a maximal root in $\Delta_{0, \mathrm{en}}$. Define $\alpha_0 = - \beta_0 + \delta_0$, which may not be a root $\TT$.  Let 
$\alpha_1, \ldots, \alpha_q$ be simple roots of $\Delta(\fma(\bar{0}, \bar{0}), \fmh(0))$ and define
the lattice $Q^+ = \displaystyle{\bigoplus_{i =0}^{q} \NN \alpha_i}$. Let for $\Lambda, \lambda \in \HH^*$, we define an ordering on $\HH$ by $\lambda \leq \Lambda$ if $\Lambda - \lambda \in Q^+ 
$. It is easy to see that for $\Lambda \leq \lambda$, $\Lambda(d_i) = \lambda (d_i)$ for $1 \leq i \leq n$.
Consider the natural triangular decomposition of $\mathcal{T}$:
\begin{align*}
& LT^{+} =  \displaystyle{ \bigoplus_{\alpha + k_0 \delta_0 >0, \,\, k \in \ZZ^{n}}}{\gg(\bar{k}_0, \bar{k}, \alpha) \otimes t_{0}^{k_0}t^{k}} ;\\
& LT^{-} =  \displaystyle{ \bigoplus_{\alpha + k_0 \delta_0 <0, \,\, k \in \ZZ^{n}}}{\gg(\bar{k}_0, \bar{k}, \alpha) \otimes t_{0}^{k_0}t^{k}} ;\\
&LT^{0} =  \displaystyle{ \bigoplus_{ k \in \ZZ^{n}}}{\gg(\bar{0}, \bar{k}, 0)}\otimes t^{k} ;\\
\SS^{+}_{n+1}(m_0,m) = &\{D(u,r)\,|\, u \in \CC^{n+1}, r \in \Gamma_0 \oplus \Gamma, (u|r) =0, r_0 > 0\} ;\\
\SS^{-}_{n+1}(m_0,m) = &\{D(u,r)\,|\, u \in \CC^{n+1}, r \in \Gamma_0 \oplus \Gamma, (u|r) =0, r_0 < 0\} ;\\
\SS^{0}_{n+1}(m_0,m) = &\{D(u,r)\,| \,u \in \CC^{n+1}, r \in \Gamma_0 \oplus \Gamma, (u|r) =0, r_0 = 0\} ;\\
& Z^{+} =   \displaystyle{ \bigoplus_{\substack{0 \leq i \leq n \\ 0 < s_0 \in \Gamma_0, s \in \Gamma}}}{\CC t_{0}^{s_0} t^{s} K_i} ;\\
& Z^{-} =  \displaystyle{ \bigoplus_{\substack{0 \leq i \leq n \\ 0>s_0 \in \Gamma_0, s \in \Gamma}}}{\CC t_{0}^{s_0} t^{s} K_i} ;\\
& Z^{0} =  \displaystyle{\bigoplus_{0 \leq i \leq n, s\in \Gamma}}{\CC  t^{s} K_i} .
\end{align*}
Let $\mathcal{T}^{+} = LT^{+} \oplus Z^{+} \oplus \SS^{+}_{n+1}(m_0,m)$, $\mathcal{T}^{-} = LT^{-} \oplus Z^{-} \oplus \SS^{-}_{n+1}(m_0,m)$ and $\mathcal{T}^{0} = LT^{0} \oplus Z^{0} \oplus \SS^{0}_{n+1}(m_0,m)$. Then $\mathcal{T} = \mathcal{T}^{-} \oplus \mathcal{T}^{0} \oplus \mathcal{T}^{+}$ is a trigular decomposition of $\mathcal{T}$.  We have the following
\bthm \label{th5}
The space $M = \{ v \in V \,|\, \mathcal{T}^{+}v = 0\} $ is non-zero.
\ethm
\begin{proof}
It will suffice to show the existence of $\Lambda \in P(V)$ such that $\Lambda + \beta + \delta_k \notin P(V)$ for any root $\beta + \delta_k \in P(V)$ with $\beta > 0$ and $k \in \ZZ^n$. The proof is similar to the proof of Theorem 5.2 of \cite{PSER}.
\end{proof}

It follows that $M$ is a irreducible $\TT^{0}$-module and $U(\TT^-)M = V$. Note that $M$ is $\ZZ^n$-graded as 
$d_i \in \SS^{0}_{n+1}(m_0,m)$. We identify $\SS^{0}_{n+1}(m_0,m) $ with $\SS_{n}(m) \oplus  \displaystyle{\sum_{r \in \Gamma}}{t^r d_0}$, where $\SS_n(m) = \{D(u,r) : u \in \mathbb{C}^n, r \in \Gamma, (u|r) = 0\}$. So we identify $\mathcal{T}^{0}$ with 
 $ LT^{0} \oplus \SS_{n}(m) \oplus  \displaystyle{\sum_{r \in \Gamma}}{\CC t^{r}d_0} \oplus Z^{0}$.

Now we will prove that weight spaces of $M$ are uniformly bounded and 
$\{t^s K_i | s \in \Gamma,  1 \leq i \leq n\}$ acts trivially on $M$. For this let us fix an $i, 1 \leq i \leq n$ and  consider the loop algebra $\fma(\bar{0}, \bar{0}) \otimes \C[t_{i}^{m_i}, t_i^{- m_i}] \oplus d_i$.
Let $\theta$ be the highest root of $\fma(\bar{0}, \bar{0})$  and $\theta^{\vee}$ be the coroot. Now as
$\fmh(0)$ acts as scalar on $M$, let $\lambda \in \fmh(0)^*$ such that $h (v) = \lambda(h) v , \forall h\in \fmh(0), \forall v \in M$, here $\lambda$ is restriction of $\Lambda$ of  Theorem \ref{th5} to $\fmh(0)$. Now consider the lattice $ \mathfrak{M} =\gamma (\mathbb{Z} [\overset{\circ}{W} (\theta^{\vee})])$, where $\overset{\circ}{W}$ denote the finite Weyl group of $\fma(\bar{0}, \bar{0})$ and $\gamma : \fmh(0) \mapsto \fmh(0)^*$ is an isomorphism.
Consider the element $t_{i, h} (\lambda) = \lambda - \lambda(h)\delta_i$, where $h \in \mathbb{Z} [\overset{\circ}{W} (\theta^{\vee})]$. Note that $\lambda(K_i) = 0$. Now as $t_{i, h_1} t_{i, h_2} = t_{i, h_1 + h_2}$, we can identify the group generated by
$\{t_{i, h} : h \in \mathbb{Z} [\overset{\circ}{W} (\theta^{\vee})] \}$ with $\mathfrak{M}$. Let $r_{\lambda} = \text{min}\{ \lambda(h) : h \in \mathbb{Z} [\overset{\circ}{W} (\theta^{\vee})] , \lambda(h) >0\}.$

\blmma \label{l4}
Consider $\lambda + r\delta_i, r \in \ZZ$. Then there exists $w \in \mathfrak{M}$ such that 
$w(\lambda + r\delta_i) = \lambda + \bar{s} \delta_i$, for $0 \leq \bar{s} < r_{\lambda}$.
\elmma
\begin{proof}
See 2.4 of \cite{CG}.
\end{proof}

Now consider the Lie algebra $\widetilde{\fmh}(A(m))= \fmh(0) \otimes A(m) \bigoplus \Omega_{A(m)}/ d(A(m))$. Let $\widetilde{W}$ be the Weyl group of the corresponding FTLA
$$\fma(\bar{0}, \bar{0})  \otimes A(m) \bigoplus \Omega_{A(m)}/ d(A(m)) \bigoplus \text{Der}(A(m)).$$ As before consider the loop subalgebra $\fma(\bar{0}, \bar{0}) \otimes \C[t_{i}^{m_i}, t_i^{- m_i}] \oplus d_i$ of the FTLA.  By the discussion above we have $t_{i, h}{\lambda} = \lambda - \lambda(h)\delta_i $ an element of $\widetilde{W}_{i} \subseteq \widetilde{W}$, where $\widetilde{W}_{i}$ denote the corresponding Weyl group of  $\fma(\bar{0}, \bar{0}) \otimes \C[t_{i}^{m_i}, t_i^{- m_i}] \oplus d_i$. Then for $s_i \in \ZZ$ and
$\lambda + s_i \delta_i = \lambda + k_i \delta_i + r_i m_i \delta_i$, where $s_i = k_i + r_i m_i$, $0 \leq k_i < m_i$. Then by the Lemma \ref{l4} we get $w_i \in \widetilde{W}$ such that
$w_i (\lambda + s_i \delta_i) = \lambda + k_i \delta_i + \bar{s_i} \delta_i$, where $0 \leq k_i < m_i$ and $0 \leq \bar{s_i} < r_{\lambda}$.

\begin{crlre} \label{cr1}
Let $\delta_r = \sum_{i =1}^{n}{r_i \delta_i}$, where $r = (r_1, \dots ,r_n) \in \ZZ^n$. Let 
$r_i = k_i + s_i m_i, 0 \leq k_i < m_i$. Then there exists $w \in \widetilde{W}$ such that 
$w(\lambda + \delta_r) = \delta_k + \sum_{i =1}^{n}{\bar{s_i} m_i \delta_i} + \lambda$, where
$0 \leq \bar{s_i} < r_{\lambda}$.
\end{crlre}

Now consider $M = \bigoplus_{k \in \ZZ^n}{M_k}$. Let $M' = \{ M_{l} : l \in \ZZ^n, l_i = k_i + m_i\bar{s_i},  0 \leq k_i <m_i,  0 \leq \bar{s_i}  < r_{\lambda}, 1 \leq i \leq n\}$. Then clearly $M'$ is a finite set. Let $N = \text{Max  dim}\{T_s : T_s \in M'\}$.
\begin{ppsn}
The dimensions of the weight spaces of $M$ are uniformly bounded by $N$.
\end{ppsn}
\begin{proof}
Follows from Corollary \ref{cr1}.
\end{proof}

\begin{ppsn}
There are only finitely many $\widetilde{\fmh}(A(m))$-submodules of $M$.
\end{ppsn}
\begin{proof}
Any such module $S$ is generated by vectors of $M$, which can be taken from the generators of $M'$
by corollary \ref{cr1}.
\end{proof}

Now as number of  $\widetilde{\fmh}(A(m))$-modules are finite consider minimal such module say 
$S_{\text{min}}$. Clearly $S_{\text{min}}$ is an irreducible $\widetilde{\fmh}(A(m))$-module
and consider $(U(\fma(\bar{0}, \bar{0})) \oplus \Omega_{A(m)}/d_{A(m)}) S_{\text{min}}$ which is an integrable module with top part irreducible as $\widetilde{\fmh}(A(m))$-module. Now consider the irreducible quotient of above module. Notice that top part goes injectively to the quotient. It follows by \cite{R} Remark 5.5
 that $ \Omega_{A(m)}/d_{A(m)}$ acts trivially  on $S_{\text{min}}$. Now consider
$\bar{M} = \{v \in M |  \Omega_{A(m)}/d_{A(m)} v = 0\}$. Then it is easy to see that $\bar{M}$ is a
$\TT^0$-module. But this forces $M = \bar{M}$ by irreducibility of $M$ as $\TT^0$-module and we have the following:
\begin{ppsn} \label{thc}
 $\Omega_{A(m)}/d_{A(m)} $ acts trivially on $M$. 
\end{ppsn}

Using above we further identify $\mathcal{T}^{0}$ as the Lie algebra $\mathfrak{L} =
LT^{0} \oplus \SS_{n}(m) \oplus  \displaystyle{\sum_{r \in \Gamma}}{\CC t^{r}d_0} \oplus  \displaystyle{\sum_{r \in \Gamma}}{\CC t^{r}K_0}$. Then $M$ is an irreducible $\mathfrak{L}$-module.

\subsection{}

Let for $k \in \ZZ^n$ with $k = (k_1, \ldots, k_n)$, define $M_{0}' = \bigoplus_{0 \leq k_i < m_i}{M_{k}}$, and $M_{r}' = \bigoplus_{r_i \leq k_i < m_i + r_i}{M_{k}}$.
Then $M = \oplus_{r \in \Gamma}{M_{r}'}$ is $\Gamma$-graded. We have the following from \cite{GL}.
\bthm \label{thi}
 Consider $\mathfrak{L}$-irreducible module $M$ with $\Gamma$ grading. Then either $t^s K_0$ acts injectively  on $M$ for all $0 \neq s \in \Gamma$ or $t^s K_0$ acts trivially on $M$ for all $0 \neq s \in \Gamma$.
\ethm
\begin{proof}
 For $D(u,r) \in \SS_n (m)$, we have $[D(u,r), t^s K_0] = (u|s) t^{r+s} K_0$.
As in our case $n \geq 2$, for $s \in \Gamma$, $s \neq 0$, we can always find $u \in \mathbb{C}^{n+1}$ and $r \in \mathbb{Z}^{n+1}$ with $r_0 = 0$, such that $(u| r) = 0$, and
$(u|s) \neq 0$. Now using the same argument as Proposition 3.4 of \cite{GL} we get the result. Note that the additional spaces $ \text{LT}^0$ and $\sum_{r \in \Gamma}{\mathbb{C} t^{r} d_0}$ don't create any problem as they commute with $t^{s} K_0$.
\end{proof}
\bppsn \label{prpm}
Let $r, s \in \Gamma$, $k \in \ZZ^n$ and $0 \neq v \in M$. Then
\begin{enumerate}
\item[{(a)}] $t^r K_0 v \neq 0$ for all $r \in \Gamma$.
\item[{(b)}]  $t^r K_0 t^s K_0 = \lambda_{r, s}  t^{r+s} K_0$ on $M$,
where $ \lambda_{r, s} = \lambda$ for all $r \neq 0, s \neq 0, r+s \neq 0$, $\lambda_{r, -r} = \mu$ for all $r \neq 0$ and $\lambda_{0, r} = C_0$ for all $r \in \Gamma$. Further we have
$\mu C_0 = \lambda^2 \neq 0$.
\item[{(c)}] $\text{dim}\,\, M_k =  \text{dim}\,\, M_{k+r} = p_k \,\, \forall r \in \Gamma$. Suppose $v_1(k), \ldots, v_{p_k}(k)$ is a basis for $M_k,$ where $0 \leq k_i < m_i$. Let
$v_i(k+s) = \frac{1}{\lambda} v_i (k), \,\, \forall i, \,\, \forall s \neq 0$. Then $v_1 (k+s), \cdots, v_{p_k}(k+s)$ is a basis for $M_{k+s}$.
\item[{(d)}] For $0 \neq r \in \Gamma, t^{r} K_0 (v_1(k+s), \cdots, v_{p_k}(k+s)) = \lambda(v_1(k+s+r), \cdots, v_{p_k}(k+s+r)) $.
\item[{(e)}] Recall that $\fma(\bar{0},\bar{0})$ is a simple Lie algebra of fixed points with $\fmh(0)$ be its Cartan subalgebra with ${\alpha}_{1}^{\vee}, \ldots ,{\alpha}_{l}^{\vee}$ be its simple co-roots. There exists $\alpha \in \fmh(0)^{*}$ such that
$h(r)(v_1(k+s), \cdots, v_{p_k}(k+s)) = \alpha(h)(v_1(k+s+r), \cdots, v_{p_k}(k+s+r))$, $\forall r \neq 0$; $\alpha(\alpha_{i}^{\vee}) = 0$ iff $\lambda(\alpha_{i}^{\vee}) = 0$.
\item[{(f)}] $t^{r}d_0(v_1(k+s), \cdots, v_{p_k}(k+s)) = \Lambda_d(v_1(k+s+r), \cdots, v_{p_k}(k+s+r))$ for all $0 \neq r \in \Gamma,  k \in \ZZ^{n}, \Lambda_d \in \mathbb{C}$ and $d_0 (v(k)) = \Lambda(d_0)v(k)$.

\end{enumerate}
\eppsn
\begin{proof}
Let us assume that $t^{s} K_0 = 0$ on $M$ for all $0 \neq s \in \Gamma$. Then consider the highest root $\theta$ of $\fma(\bar{0},\bar{0})$ and $x_{\theta}$ and $x_{-\theta}$ be elements
of $\fma(\bar{0},\bar{0})_{\theta}$ and $\fma(\bar{0},\bar{0})_{- \theta}$ such that $\langle x_{\theta}, x_{-\theta}\rangle = 1$. Then it is easy to see that 
$x_{- \theta} \otimes t_0$ and $x_{\theta} \otimes t_0^{-1}$ and $h_{\theta} + K_0 $ form a $\mathfrak{sl}_2$ copy.  Consider 
the Lie algebra $\mathfrak{sl}_2 \otimes \mathbb{C}[t_1, t_1^{-1}]$ with bracket:
$[(x_{- \theta}\otimes t_0) \otimes t_1^{r_1}, (x_{\theta} \otimes t_{0}^{-1}) \otimes t_1^{s_1}] = (- h_{\theta} + K_0) \otimes t_{1}^{r_1 + s_1} = - h_{\theta} \otimes t_{1}^{r_1 + s_1} + K_0 \otimes t_{1}^{r_1 + s_1}$.

Now let us look at the bracket in our setting
$[x_{- \theta} \otimes t_0 t_{1}^{r_1}, x_{\theta} \otimes t_{0}^{-1} t_{1}^{s_1}] = - h_{\theta}\otimes t_1^{r_1 + s_1} + t_{1}^{r_1 + s_1} K_0 + r_1 t_{1}^{r_1 + s_1} K_1$.
But as $r_1 t_{1}^{r_1 + s_1} K_1$ acts trivially on $V$, we identify $K_0 \otimes t_1 ^{r_1 + s_1}$ with $ t_1 ^{r_1 + s_1} K_0$. 
Now as $V$ is an integrable $\TT$ module for $\TT$ in particular for $L(\mathfrak{sl}_2) = \mathfrak{sl}_2 \otimes \mathbb{C}[t_1, t_1^{-1}]$ and
every element of $M$ is highest weight vector for $L(\mathfrak{sl}_2)$. Now following the notations in \cite{CP}:
for any positive root $\beta$ of $ \fma(\bar{0},\bar{0})$, $\Lambda_{\beta}^{\pm}(u) = \sum_{m = 0}^{\infty}{\Lambda_{\beta, \pm m} u^{m}} = \mathrm{exp}\big(-\sum_{k =1}^{\infty}{\frac{h_{\beta, \pm k}}{k} u^{k}}\big)$,
where $h_{\beta} \in \fmh(0)$ and $h_{\beta, \pm k} = h_{\beta} \otimes t^{\pm k}$.  
Let $\delta_0$ be the corresponding element of $\HH^*$ with respect to $K_0$. 
$\theta^{'} = - h_{\theta} + K_0$. So $K_0 = - h_{\theta} + K_0 + h_{\theta}$.  Then by Proposition 1.1
of \cite{CP} 
we have $\Lambda_{\theta^{'}, m} v = 0$ for all $ m > \lambda(\theta')$. Similarly $\Lambda_{\theta, m} v = 0$ for all
$m > \lambda(\theta)$ and $\Lambda_{\theta^{'}, m} v \neq 0$ for $ m = \lambda(\theta')$ and
$\Lambda_{\theta, m} v \neq 0$ for $ m = \lambda(\theta)$. But as $\Lambda_{\delta_0, m} = \sum_{m_1 + m_2 = m}{\Lambda_{\theta', m_1} \Lambda_{\theta, m_2}}$.  From above it follows that
$\Lambda_{\delta_0, \lambda(k_0)} = \Lambda_{\theta', \lambda(\theta')} \Lambda_{\theta, \lambda(\theta)}$.  But  as by assumption $t^{s} K_0 = 0$ on $M$ for all
$0 \neq s  \in \Gamma$. So $\Lambda_{K_0, m} = 0$ for all $m >0$. But as RHS of the above equation is non-zero, this forces that $\lambda(K_0) = 0$ which is a contradiction. So $t^{s}K_0$ acts injectively on $M$ for all $0 \neq s \in \Gamma$.
By Theorem \ref{thi} we have
$t^{s}K_0 v \neq 0$, for any nonzero $v \in M$ and $\forall s \in \Gamma$ (for $0 =s \in \Gamma, K_0$ acts by non-zero scalar $C_0$) and $\mathrm{dim}M_{k} = \mathrm{dim}M_{k+s} = p_{k}$, $\forall s \in \Gamma$.
If $\{v_1, \ldots, v_{p_{k}}\}$ is a basis of $M_k$ then $\{v_1(k+s), \ldots, v_{p_k}(k+s)\}$ is a basis for
$M_{k+s}$, where $\frac{1}{\lambda}t^{s}K_0 v_i (k) = v(k+s)$ for all $0 \neq s \in \Gamma$ and $t^{s}K_0(v_1 (k+r), \ldots, v_{p_{k}}(k+r)) = \lambda(v_1 (k+r+s), \ldots, v_{p_{k}}(k+r+s)), \,\, \forall s \in \Gamma, s \neq 0$. It follows that $M \cong \Bigg( \displaystyle{\bigoplus _{\substack{0 \leq k_i < m_i \\ 1 \leq i \leq n}}}{M_{k}}\Bigg) \otimes A(m)$. Let us denote $M^1 = \displaystyle{\bigoplus _{\substack{0 \leq k_i < m_i \\ 1 \leq i \leq n}}{M_{k}}}$ and
denote $v(0) = v$ for $v \in M^1$. Let $\frac{t^{r}K_{0}}{\lambda}v(0) = v(k), r \neq 0, r \in \Gamma$.
Now (b) follows trivially. From
(b) it follows that $\frac{t^{s}}{\lambda} v(r) = v(r+s)$.
Now for fixed $i$, let $\alpha_i^{\vee}$ be a coroot of $\fmh(0)$. As $\Omega_{A(m)}/d_{A(m)}$ acts trivially on $M$ by Proposition \ref{thc}, $\fmh(0) \otimes A(m)$ is abelian on $M$ and as $D(u, r) \alpha_i^{\vee} (s) = (u|s) \alpha_i^{\vee}(s+r)$, from \cite{GL} it follows that either $\alpha_i^{\vee} (s)$ is zero for all $s \neq 0, s \in \Gamma$ or  $\alpha_i^{\vee} (s)$ is injective for all $s \neq 0, s \in \Gamma$ . But as in previous case using \cite{CP} it follows that $\alpha_i^{\vee}$ acts trivially when $\alpha_i^{\vee}(s)$ acts trivially for all $0 \neq s \in \Gamma $. Hence we can assume that 
$\alpha_i^{\vee} (s)$ acts injectively for all $s \in \Gamma$. Now consider a non-zero scalar $b_i \in \mathbb{C}$.
Then $D(u,r)(\alpha_i^{\vee} (s) - b_i t^s K_0) = (u|s)(\alpha_i^{\vee} (s) - b_i t^s K_0)$ for $u \in \mathbb{C}^n, r, s \in \Gamma, r,s \neq 0, (u|r) = 0$. Then again by \cite{GL} it follows that
$\alpha_i^{\vee} (s) - b_i t^s K_0$ is either zero or injective for all $s \in \Gamma$. Now for $0 \neq r \in \Gamma$,
consider the nonzero operator $t^{-r}K_0 \alpha_{i}^{\vee}(r)$ on $M_{0}$. Then there exists $0 \neq v \in M_0$ such that 
$t^{-r}K_0 \alpha_{i}^{\vee}(r) v = \lambda_i v $ for some $0 \neq \lambda_i \in \mathbb{C}$. Then operating 
$t^{r}K_0$ both side we get $\mu C_0 \alpha_i^{\vee} (r)  v = \lambda_i t^{r} K_0 v$. Now since $\mu C_0 = \lambda ^{2} \neq 0$, we get $\alpha_i^{\vee}(r)v - \frac{\lambda_i}{\mu C_0} t^r K_0 v = 0$ and so by earliar discussion we have
$\alpha_i^{\vee}(r)v = \frac{\lambda_i}{\mu C_0} t^r K_0 v = \frac{\lambda_i \lambda}{\mu C_0}v(r) = \frac{\lambda_i}{\lambda}v(r)$  for all $0 \neq r \in \Gamma$. Now defining $\alpha : \fmh(0) \mapsto \CC$ by $\alpha(\alpha_{i}^{\vee} ) = \frac{\lambda_i}{\lambda}$ for $1 \leq i \leq n$, we get the desired result. Similarly (f) follows as $[D(u,s), t^r d_0] = (u | r) t^{r+s}d_0$, hence by similar argument as
(e), we get the desired result.
\end{proof}
\subsection{}
 Now using $(d)$ of above proposition we  identify $ t^{r}K_0$ with $\frac{1}{\lambda}t^{r}$ for $r \in \Gamma$.  Hence $M^{1} \otimes A(m)$ becomes an irreducible module for the 
space $\mathfrak{L} = LT^{0} \oplus \SS_{n}(m) \oplus  \displaystyle{\sum_{r \in \Gamma}}{\CC t^{r}d_0} \oplus A(m)$.
Now we will rewrite $LT^{0}$ for our convenience: for that let
$\gg(0) = \{X \in \fma \,\,\, |\,\,\, \sigma_0(X) = X, [h, X]= 0, h \in \fmh(0)\}$. Then $\gg(0) \neq 0$ as $\fmh(0) \subseteq
\gg(0)$ and $\sigma_i(\gg(0)) \subseteq \gg(0) $. So $\gg(0)$ preserves $\Lambda$- grading and let
$\gg(0) = \displaystyle{\bigoplus_{\bar{k} \in \Lambda}}{(\gg{(0)})_{\bar{k}}}$ and let
$L\big(\gg(0), \sigma\big) = \displaystyle{\bigoplus_{k \in Z^{n+1}}}{(\gg(0))_{\bar{k}}\otimes t^{k}}$ be the corresponding multiloop algebra. Then it is easy to see that $LT^{0} = L\big(\gg(0), \sigma\big)$.
Now consider the subspace $W$ of $M = M^{1} \otimes A(m)$ spanned by $\{v(r)- v(0) : v\in M^{1}, r \in \Gamma \}.$
Then $W$ is module for $L\big(\gg(0), \sigma\big) \rtimes \SS_{n} '(m)\ \oplus A(m) \oplus \displaystyle{\sum_{r \in \Gamma}}{\CC t^{r}d_0}$, where $\SS_{n}'(m) = \mathrm{span} \{D(u,r) - D(u,0) : u \in \mathbb{C}^n, r\in \Gamma, (u|r) =0\}$. Let us consider $\widetilde{V} =( M^{1} \otimes A(m))/W$. By previous argument 
 $M/W$ is a finite dimensional $ L\big(\gg(0), \sigma\big) \rtimes \SS_{n} '(m)\oplus A(m) \oplus \displaystyle{\sum_{r \in \Gamma}}{\CC t^{r}d_0}$.
 As $ A(m) \oplus  \displaystyle{\sum_{r \in \Gamma}}{\CC t^{r}d_0}$ act as scalars on $\widetilde{V}$,
we neglect them and consider $\widetilde{V}$ as $\mathcal{I} :=  L\big(\gg(0), \sigma\big) \rtimes \SS_{n} '(m)$-module.
Now define $L(V') = V' \otimes A_{n}$, where $V'$ is any $\I$ module. Let fix $\gamma \in P(V)$, where $V$ is an integrable irreducible module for $\mathcal{T}$ with finite dimensional weight spaces. Let $\beta_i = \gamma(d_i)$ for $1 \leq i \leq n$, let $\beta = (\beta_1, \ldots \beta_n) \in \mathbb{C}^n$. Let $\mu \in \mathbb{C}^n$ be arbitrary. Let $V'$ be any $\I$-module. Then $L(V')$ is a $\mathfrak{L}$-module by the following
action:
\begin{align*}
X(k).v \otimes t^{s} & =   (X(k)v) \otimes t^{k+s} ;\\
0 \neq r \in \Gamma, D(u,r).(v \otimes t^{s}) & =   (I(u,r)v) \otimes t^{r+s} + (u| s+\mu)v \otimes t^{s+r};\\
t^{l}.v \otimes t^s  & =  v \otimes t^{s+l};\\
t^{l}d_0. v \otimes t^{s} & = \gamma(d_0).v \otimes t^{t+s}\\
D(u, 0) v \otimes t^{l} & = (u|\beta) v \otimes t^{l}\\
\frac{t^{r}}{\lambda}K_{0}. v \otimes t^{l} & = v \otimes t^{r +l}, r \neq 0, r\in \Gamma \\
\frac{K_0}{C_0}v \otimes t^{l} & = v \otimes t^{l}.
\end{align*}
where $v \in V, s ,l \in \Gamma, u\in \CC^{n}, r\in\,\,  \Gamma \,\,\, \text{with}\,\, (u|r) = 0$.

It is easy to check that $L(V')$ is an $\mathfrak{L}$-module (see 8.3 of \cite{PSER}). Now we take $\widetilde{V} = M/W$. Then by above $L(\widetilde{V})$ is an $\mathfrak{L}$-module. Consider a map $\widetilde{\phi}: M \mapsto L(\widetilde{V})$ by $\widetilde{\phi}(v_k) = \bar{v}_{k}\otimes t^{k}, k \in \mathbb{Z}^{n}$, where
$v_{k} \in M_{k}, \bar{v}_{k}\in M/W$. It follows that $\widetilde{\phi}$ is a  non-zero $\mathfrak{L}$-module map. Since $M$ is an irreducible $\LL$-module, it follows that $\widetilde{\phi}$
is injective. In particular $\widetilde{\phi}(M) \subseteq L(\widetilde{V})$. Let us define for a fixed $p \in \Lambda$. 
$L(\widetilde{V})(\bar{p}) = \{ \bar{v}_{k} \otimes t^{k+r+p},  r \in \Gamma, k \in \mathbb{Z}^{n}\}$. Then $L(\widetilde{V}) = \bigoplus_{\bar{p} \in \Lambda}{L(\widetilde{V})(\bar{p})}$
Then $M \cong L(\widetilde{V})(\bar{0})$. In particular $L(\widetilde{V})(\bar{0})$ is an irreducible $\LL$-module (see section 8 of \cite{PSER} for details).
\section{} \label{sec7}
\subsection{}
 In order to get hold of $M$, by previous section, we need to understand $\widetilde{V}$ as $\mathcal{I}$-module. We devote this section for this purpose. 
Let $\gg(0) = \gg(0)_{\text{ss}} \oplus R$, where $\gg(0)_{\text{ss}}$ and $R$ are Levi  and
radical part of $\gg(0)$. Then as $\sigma_i (\gg(0)_{\text{ss}}) = \gg(0)_{\text{ss}}$ and $\sigma_i (R) = R$ for $1 \leq i \leq n$, we have $L(\gg(0), \sigma) = L(\gg(0)_{\text{ss}}, \sigma) \oplus L(R, \sigma)$.  Now we invoke result from \cite{PSER}
\bthm
$\widetilde{V}$ is completely reducible as $\mathcal{I}$ module with isomorphic irreducible components.
\ethm
\begin{proof}
The proof of this theorem is exactly similar to Theorem 8.3 of \cite{PSER}. 
\end{proof}
In the light of above proposition it is enough to consider an irreducible representation of $\I$.
Let $(W, \pi)$ be a finite dimensional irreducible representation of $L(\gg(0), \sigma) \rtimes \SS_n'(m) = \mathcal{I}$.
Let $\pi(L(\gg(0), \sigma)) = \gg''$. As $\pi(\mathcal{I})$ is reductive with atmost one dimensional center, there exists a unique compliment $\gg$ in $\pi(\mathcal{I})$. So $W$ is an irreducible module for $\gg' \oplus \gg''$. Then it is standard that $W \cong W_1 \otimes W_2$, where $W_1$ and $W_2$ are irreducible modules for $\gg'$ and $\gg''$ respectively. Now we concentrate on $W_2$ as $L(\gg(0), \sigma) = L(\gg(0)_{\text{ss}}, \sigma) \oplus L(R, \sigma)$. Now as $R$ is solvable ideal, it follows that $\pi(L(R, \sigma))$
lies in the center of $\pi(\mathcal{I})$ which is atmost one dimensional. Hence $L(R, \sigma)$ acts as a scalar on $W$. In particular $V_2$ is an irreducible module for $L(\gg(0)_{\text{ss}}, \sigma)$. Now we recall result from \cite{NSS} for which we need to introduce some notations: For each $i, 1 \leq i \leq n$, and a fixed positive integer $l$, let $\underline{a}_i = (a_{i ,1}, a_{i, 2}, \cdots, a_{i ,l})$ such that
$a_{i ,j}^{m_i} \neq a_{i ,t}^{mi}$ for $j \neq t$ (**). Let $\bar{\gg}$  be a finite dimensional semisimple Lie algebra. Let $\sigma_1, \sigma_2, \cdots, \sigma_n$  be finite order automorphisms on $\bar{\gg}$ of order $m_1, m_2, \cdots m_n$  respectively. Let $L(\bar{\gg}, \sigma)$ be corresponding multiloop
algebra. Let $I = \{(i_1, i_2, \ldots, i_n)| 1 \leq i_j \leq l \}$. Let for $S \in I$, $S = (i_1, i_2, \ldots, i_n)$ and $r = (r_1, r_2, \cdots, r_n) \in \ZZ^{n}$, $a_{S}^{r} = a_{1, i_1}^{r_1}a_{2, i_2}^{r_2}\cdots a_{n, i_n}^{r_n}$. Now consider the evaluation map
$\phi:\bar{\gg} \otimes A \mapsto \bigoplus{\bar{\gg}} ({nl} \,\, \text{copies})$, $\pi(X \otimes t^{r})=
(a_{I_1}^{r}(X), a_{I_2}^{r}(X)), \ldots, a_{I_{nl}}^{r}(X)))$, where $I_1, I_2, \ldots, I_{nl}$ is some order on
$I$. Consider restriction of $\phi$ to $L(\bar{\gg}, \sigma)$. We have the following theorem which follows from \cite{NSS}:
\bthm
 Let $W'$ be a finite dimensional irreducible representation of $L(\bar{\gg}, \sigma)$. Then represention factors though $\bigoplus{\bar{\gg}} ({nl} \,\, \text{copies})$.
\ethm
\begin{proof}
See \cite{NSS} Corollary 6.1 and the discussion thereafter. Note that some factors in  $\bigoplus{\bar{\gg}} ({nl} \,\, \text{copies})$ may act trivially and we discard them.
\end{proof}
From above result it follows that $\bigoplus{\gg(0)_{\text{ss}}} \cong \gg''_{\text{ss}}$ as both can realised as the quotient $L({\gg(0)_{\text{ss}}}, \sigma)/\text{Ker}(\pi_1)$. So we have 
$\pi(L(\gg(0), \sigma) \rtimes \SS_n'(m)) = \oplus{\gg(0)_{\text{ss}}} \bigoplus \mathbb{C} \bigoplus \gg''$.
Now we choose $i$-th  piece of $\bigoplus{\gg(0)_{\text{ss}}}$ and choose projection of the map $\pi$,
say $\pi_i$ onto it, i.e., $\pi_i (L(\gg(0), \sigma) \rtimes \SS_n'(m)) \mapsto \gg(0)_{\text{ss}}$.
Then we claim that $\pi_i (S_n' (m)) = \{0\}$. For consider $[D(u,r) - D(u,0) , X(k+s)] = (u| k+s) X(k+s+r) - X(k+s)$, where $k \in \ZZ^{n}, s, r \in \Gamma$. Apply $\pi_i$ both side we get
$[\pi_i (D(u,r) - D(u,0)) , a_{I_{i}}^{k+s}X] = (u | k+s)a_{I_{i}}^{k+s}(X)(a_{I_{i}}^{r} - 1)$. 
Now cancelling $a_{I_{i}}^{k+s}$ both side we get
$[\pi_i (D(u,r) - D(u,0)) , X] = (u | k+s)(a_{I_{i}}^{r} - 1)(X).$ Now as LHS is independent of $s$, this forces that $\pi_i (D(u,r) - D(u,0)) = 0$. We deduce that $a_{I_{i}}^{r} = 1$ for all $r \in \Gamma$. From this it
follows that $\pi_i (X(k+s) - X(k)) = 0$. As these equations take place in semisimple Lie algebra which has no center, hence the claim follows.
Now in view of condition (**) we see that the representation of $L(\gg(0)_{\text{ss}}, \sigma)$ factors through only one copy of $L(\gg(0)_{\text{ss}}$, i.e., $\gg(0)_{\text{ss}} \cong \gg'_{\text{ss}}$.
\subsection{}
Now we turn our attention towards finite dimensional  irreducible modules for $\SS_{n}'$. For clarity
we once again recall the notations: $A_n = A = \mathbb{C}[t_1^{\pm1}, \cdots, t_n^{\pm 1}]$, and
$\SS_n = \{ D(u, r) : u \in \mathbb{C}^n, r \in \ZZ^n, (u|r) = 0\}$, where $D(u,r) = \sum_{i =1}^{n}{u_i t^{r} t_i \frac{d}{dt_i}}$, where $u = (u_1, \ldots, u_n) \in \mathbb{C}^n$ and $r = (r_1, \ldots, r_n)\in \ZZ^{n}$ and $\SS_n ' = \text{span} \{D(u,r) - D(u,0)| u \in \mathbb{C}^n, r \in \ZZ^n, (u|r) =0\}$. Then $\SS_n '$ is a Lie algebra with brackets: $[I(v,s), I(u,r)] = (u,s)I(v,s) - (v,r)I(u,r) + I(w, r+s)$ where $w = (v,r)u - (u,s)v, v,u \in \mathbb{C}^n,  r, s \in \ZZ^{n}$ and $I(u,r) = D(u,r) - D(u,0)$. Now consider
the Lie algebra $\SS_n \ltimes A$. Our aim is try to understand the relation between finite dimensional $\SS_n '$ module and modules with $\SS_n \ltimes A$ with finite dimensional weight spaces. For $\SS_n \ltimes A$, the subalgebra $\{d_i = t_i \frac{d}{dt} : 1 \leq i \leq n\}$ plays the role of a Cartan subalgebra.  A module for $\SS_n \ltimes A$ we mean $A$ action is associative. In particular $t^0 = 1$
acts as identity. The action of $t^r$ is invertible for $r \in \ZZ^n$. We recall a Theorem of \cite{BT}. Let $\tilde{V}$ be an irreducible module for $\SS_n  \otimes A$ with finite dimensional weight spaces. Let $W'$ be an $\mathfrak{sl}_n$
module and extend to $\mathfrak{gl}_n$ by letting $I$ act trivially. Let $E_{ab}$ be generators for 
$\mathfrak{gl}_n$.  Let $I = \sum_{i =1}^{n} {E_{aa}}$. 
Note: Suppose $I(v,s)$ is replaced by $I(v,s) + (v,\gamma)$ for some $\gamma \in \mathbb{C}^n$. Then these elements also satisfy above bracket. We make the following observations:
\begin{enumerate}
\item Suppose $W$ is a finite dimensional module for $\SS_n '$. Then $(\pi_{\alpha, \beta}, W \otimes A = L(W))$ is a $\SS_n \ltimes A$ module in the following way:\\
for $r \neq 0, D(u,r)(w \otimes t^{k}) = (I(u,r) w) \otimes t^{r + k} + (u| \beta + k) w \otimes t^{r+k}$ ;\\
$D(u,0)(w \otimes t^{k}) = (u| \alpha + k) w \otimes t^{k}$.
\item Suppose $V = \bigoplus_{k \in \ZZ^n}{V_{\alpha + k}}$ is an weight module for $S_n \ltimes A$ such that
$D(u,0)(w \otimes t^{k}) = (u |\alpha + k) w \otimes t^{k}, v \in V_{\alpha + k}$. Put $t^{r} V_{\alpha} = V_{\alpha + r}$. Then $V = V_{\alpha} \otimes A$.
$D(u,0) v\otimes t^{k} = (u| \alpha + k) v\otimes t^{k}, v \in V_{\alpha}$. Let $W = \{ v \otimes t^k - v(0)| v\in V_{\alpha}\}$. It can be checked to be $\SS_n '$-module.
Now consider a $\SS_n '$-module $\overline{V} = \frac{V_{\alpha} \otimes A }{W}$. Let us denote it by $(\theta, \overline{V})$. Define $(\theta_{\gamma}, \overline{V})$ $\SS_n '$-module
$(\theta_{\gamma} I(u,r))v = (\theta(I(u,r))v + (u, \gamma) v$ for $\gamma \in \mathbb{C}^n$. Thus 
$\theta_0 = \theta$. 

The following can be checked:
\begin{enumerate}
\item Suppose $(\theta, W)$ is a $\SS_n '$-module., where $W$ is finite dimensional. Let $(\pi_{\alpha, \beta}, L(W))$ is a $\SS_n \ltimes A$-module. Then $(\theta_{\alpha - \beta} , \overline{L(W)}) \cong (\theta, W)$ as $\SS_n '$-module.

\item Suppose $(\pi, W \otimes A)$ is $\SS_n \ltimes A$-module with $\text{dim} W < \infty$ with
$D(u,0) v(k) = (u| \alpha + k) v(k)$.  Let $(\theta_{\mu}, \overline{W \otimes A})$ be $\SS_n '$ - module.
Then $(\pi_{\alpha, \alpha- \mu} , L(\overline{W \otimes A})) \cong (\pi, W \otimes A)$ as 
$\SS_n \ltimes A $-module.
\end{enumerate}

\end{enumerate}
The following are easy to prove:
\begin{lmma}
Suppose $(\theta, W)$ is a finite dimensional irreducible $\SS_n'$- module. Then $(\pi_{\alpha, \beta}, L(W))$ is an irreducible $\SS_n \ltimes A$-module.
\end{lmma}

\begin{lmma}
Suppose $(\pi, L(W)$ is an irreducible $\SS_n \ltimes A$- module  with dim $W < \infty $. Then $(\theta_{r}, \overline{L(W)})$ is an irreducible finite dimensional module for $\SS_n '$
\end{lmma}

\begin{thm} \label{ths}
Let $W$ be  finite dimensional irreducible module for $\mathfrak{sl}_n$ and extend it to $\mathfrak{gl}_n$ by letting $I$ trivially. Let $E_{i j}$ be generators of $\mathfrak{gl}_n$. Then $W$ can be made into $S_n '$-module by the action:
$I(u,r) w = \sum_{i, j}{u_i r_j E_{j i} w} + (u, \gamma) w$, where $\gamma \in \mathbb{C}^n$. 
\end{thm}

Then $W$ is irreducible as $\SS_{n} '$-module,  as  $(\pi_{\alpha, \beta}, L(W))$ is an ireducible $\SS_n \ltimes A$ for any $\alpha, \beta \in \mathbb{C}^n$ .
We will now recall a result from \cite{BT} where they have classified irreductible $S_n \ltimes A$ modules
with finite dimensional weight spaces. Using this it follows that  all irreducible finite dimensional $S_n '$-modules occur as in Theorem \ref{ths} .

\begin{thm}
Suppose $W$ is a finite dimensional irreducible $\mathfrak{sl}_n$-module (extend $W$ to $\mathfrak{gl}_n$ trivially as before). Let $\alpha, \beta \in \mathbb{C}^n$. Consider $W \otimes A$ and the
action $D(u,r) (w \otimes t^{k}) = (u, k+\beta) w \otimes t^{k+r} + \sum_{i,j}{u_i r_j E_{ji}w \otimes t^{k+r}},$ for $r \neq 0$ and 
$D(u,0)(w \otimes t^k) = (u, \alpha + k) w \otimes t^k$ and $t^{r}(w \otimes t^k) = w \otimes t^{k+r}$. And all irreducible representations of $S_n \ltimes A$ module with finite dimensional weight spaces occur in this way.
\end{thm}

We summerise the results of this section. We have proved that $\pi(S_n ') \supseteq \mathfrak{g}'$ and $\mathfrak{g}' = \mathfrak{sl}_n {(\mathbb{C})}$. So $\pi(\widetilde{L}) = \mathfrak{sl}_n (\mathbb{C}) \oplus \overset{\circ}{\mathfrak{g}_{\text{ss}}} \oplus \mathbb{C} c$. Now take a finite dimensional irreducible module for $\mathfrak{sl}_n (\mathbb{C}).$ Note that  $\overset{\circ}{\mathfrak{g}_{\text{ss}}}$ is $\Lambda$-graded. Let $W_2$ be finite dimensional $\Lambda$-graded
irreducible module for $\overset{\circ}{\mathfrak{g}_{\text{ss}}} $ such that it is compatible with respect
to $\Lambda$- grading of $\overset{\circ}{\mathfrak{g}_{\text{ss}}} $ (in reality the gradation is given by the quotient of $\Lambda$ but we assume it is $\Lambda$-graded to avoid additional notations; see section 8 of \cite{PSER}). Recall that $\widetilde{V} = \bigoplus {M_i}$, where all $M_i$ 's are isomorphic as $\widetilde{L}$-modules. Each $M_i$ is isomorphic to $V' \otimes V''$ as $\mathfrak{sl}_n \oplus \overset{\circ}{\mathfrak{g}_{\text{ss}}} $  module. Then $\widetilde{V} = V' \otimes \sum{ V"}$.

\section{Description of modules for $\TT$} \label{sec8}

Recall that $\TT^{0} = L(\fma(0), \sigma)\oplus Z^{0} \oplus \SS^{0}_{n+1}(m_0,m) $, where $ Z^{0} =  \displaystyle{\bigoplus_{0 \leq i \leq n, s\in \Gamma}}{\CC  t^{s} K_i}$ and $\SS^{0}_{n+1}(m_0,m)$ is identified with $\SS_{n}(m) \oplus \displaystyle{\sum_{r \in \Gamma}}{\CC t^{r}d_0}$.Now let, $W_1$ and $W_2$ be irreducible modules for $\mathfrak{sl}_n$ extended trivially to $\mathfrak{gl}_n$ and $\gg(0)$ respectively with $W_2$ being $\Lambda$-graded which is compatible with $\Lambda$-gradation of $\gg(0)$.
Define $\TT^{0}$ action on $W_1 \otimes W_2 \otimes A_n$ as follows:
For $\alpha, \beta \in \mathbb{C}^n, r \neq 0,  r \in \Gamma,  \gamma \in \fmh(0)^*, \Lambda in \HH^{*}$ and $\lambda\in \fmh(0)^*$ such that $\Lambda |_{\fmh{(0)}} = \lambda $, \\\
$D(u,r) (w_1 \otimes w_2 \otimes t^k) = (u, k+\beta) w_1 \otimes w_2 \otimes t^{k+r} + (\sum_{i, j}{u_i r_j E_{j i} w_1})\otimes w_2 \otimes t^{k+r}$.
For $r = 0$, 
$D(u,0) (w_1 \otimes w_2 \otimes t^k) = (u, k+\alpha) w_1 \otimes w_2 \otimes t^k $.\\
 $X \in \overset{\circ}{\mathfrak{g}}_{\text{ss}, \bar{k}}$, $X\otimes t^{k} (w_1 \otimes w_2 \otimes t^{l}) = w_1 \otimes X w_2 \otimes t^{k+l}, k , l \in \mathbb{Z}^n$. \\
 $r \neq 0, r \in \Gamma$, $\frac{1}{\lambda} t^{r} K_0 (w_1 \otimes w_2 \otimes t^{l}) = w_1 \otimes w_2 \otimes t^{k+l}$;\\ $K_0 (w_1 \otimes w_2 \otimes t^{k}) = 
 C_0 (w_1 \otimes w_2 \otimes t^{k})$.  \\For $r \in \Gamma, r \neq 0$, $d_0 \otimes t^{r} (w_1 \otimes w_2 \otimes t^{l}) = \Lambda_{d} (w_1 \otimes w_2 \otimes t^{l + r})$, where $\Lambda_d \in \CC$;\\
  $r = 0$, $d_0 (w_1 \otimes w_2 \otimes t^{l}) = \Lambda(d_0) w_1 \otimes w_2 \otimes t^{l}$.\\
  $r \neq 0, h \in \mathfrak{h}(0),
h \otimes t^{r}(w_1 \otimes w_2 \otimes t^{l}) = \gamma(h)(w_1 \otimes w_2 \otimes t^{l+r})$,  $r = 0,
h (w_1 \otimes w_2 \otimes t^{l}) = \lambda(h) w_1 \otimes w_2 \otimes t^{l};\, \lambda(h) = 0 \,\,\text{iff}\,\, \gamma(h) = 0$. Now take a one dimensional representation of $L(R, \sigma)$ say $\psi$. Then $L(R, \sigma)$ acts as a $\psi$.
More precisely for $y \otimes t^k$, $y \in R_{\bar{k}}$, Then $y \otimes t^{k}(w_1 \otimes w_2 \otimes t^{l}) = \psi(y)(w_1 \otimes w_2 \otimes t^{k+l}) $. One can check that $W_1 \otimes W_2 \otimes A_n$ is $\gg(0) \otimes A_n \oplus  Z^{0} \oplus \SS^{0}_{n+1}(m_0,m)$. It has $\TT^{0}$- module structure by considering $\Lambda$ gradation of $W_2 = \displaystyle{ \bigoplus_{\bar{k}\in\Lambda}{W_{2, \bar{k}}}}$ which is compatible with 
$\Lambda$-gradation of $\gg(0)$. Then it follows that the submodule $M' =\displaystyle{\bigoplus_{k \in \ZZ^{n}}}{W_1 \otimes W_{2, \bar{k}} \otimes t^{k}}$ is an irreducible module for $\TT^{0}$. Consider the induced module 
$\mathcal{M} = \mathrm{Ind}_{\TT^{0} \oplus \TT^{+}}{M'}$, with $\TT^{+}$ acting trivially on $M'$. Let $\mathcal{M}^{\mathrm{rad}}$ be the unique maximal submodule of $\mathcal{M}$. Then $\mathcal{M}/\mathcal{M}^{\mathrm{rad}}$ is an
ireducible module for $\TT$. We have the following:

\bthm \label{THM}
Let $V$ be an irreducible integrable $\TT$ module with finite dimensional weight spaces, with $K_0$ acting as $C_0 \in \ZZ_{>0}$ and rest of $K_i$'s act trivially for $1 \leq i \leq n$. Then $V \cong \mathcal{M}/\mathcal{M}^{\mathrm{rad}}$.
\ethm

\brmk
Using $\Lambda$-gradation on $\widetilde{V}$ one can define $\Lambda$-gradation on $L(\widetilde{V})$. In case if $\widetilde{V}$
is irreducible as $\mathcal{I}$-module, then it follows that $M$ is isomorphic as $\TT^{0}$-module to the zeroth graded piece of $L(\widetilde{V})$ with respect to $\Lambda$-gradation. When $\widetilde{V}$ is reducible as $\mathcal{I}$-module, inclusion of $M$ in
$L(\widetilde{V})$ is more complex. Infact there are several copies of $M$ inside $L(\widetilde{V})$. See Section 8 of \cite{PSER} for more details.
\ermk

%\bibliography{rnewweyl.bib}

\nocite*{}

School of mathematics, Tata Institute of Fundamental Research,
Homi Bhabha Road, Mumbai 400005, India.\\
email: sena98672@gmail.com, senapati@math.tifr.res.in

Department of mathematics and statistics, IIT Kanpur,
Kalyanpur, Kanpur, 208016, India.\\
sachinsh@iitk.ac.in

Harish-Chandra Research Institute,
HBNI, Allahabad 211019, India.\\
batra@hri.res.in
\end{document}